\documentclass[reqno,a4paper]{amsart} 
\usepackage{amsmath}
\usepackage{amssymb,amsfonts} 
\usepackage{amsthm} 
\usepackage{enumerate}
\usepackage[mathscr]{eucal} 
\usepackage{eqlist}
\usepackage{cite}
\usepackage{hyperref}
\usepackage{supertabular}
\usepackage{xcolor}

  \theoremstyle{plain} \newtheorem{theorem}{Theorem}[section] 
\newtheorem{lemma}{Lemma}[section]   \theoremstyle{definition} 
  \theoremstyle{remark} 
\numberwithin{equation}{section}     
   \theoremstyle{theorem}

\begin{document}

\title[An elementary app. to the gen. Ramanujan-Nagell eq.] {An elementary approach to the generalized Ramanujan-Nagell equation}

\author[E. K. Mutlu, M. Le and G. Soydan]{El\.If K{\i}z{\i}ldere Mutlu, Maohua Le and G\"{o}khan Soydan}

\address{{\bf Elif K{\i}z{\i}ldere Mutlu}\\ Department of Mathematics \\ Bursa Uluda\u{g} University\\
 16059 Bursa, T\"URK\.IYE}
\email{elfkzldre@gmail.com}

\urladdr{http://orcid.org/0000-0002-7651-7001}

\address{{\bf Maohua Le}\\
	Institute of Mathematics, Lingnan Normal College\\
	Zhangjiang, Guangdong, 524048 China}

\email{lemaohua2008@163.com}
\urladdr{http://orcid.org/0000-0002-7502-2496}

\address{{\bf G\"{o}khan Soydan} \\ 	Department of Mathematics \\ 	Bursa Uluda\u{g} University\\ 	16059 Bursa, T\"URK\.IYE} \email{gsoydan@uludag.edu.tr }
\urladdr{http://orcid.org/0000-0002-6321-4132}

\newcommand{\acr}{\newline\indent}

\thanks{}


\begin{abstract}
Let $k$ be a fixed positive integer with $k>1$. In this paper, using various elementary methods in number theory, we give criteria under which the equation $x^2+(2k-1)^y=k^z$ has no positive integer solutions $(x,y,z)$ with $y\in\{3,5\}$.
\end{abstract}

\maketitle

\section{Introduction}\label{sec:1}
Let $\mathbb{Z}$, $\mathbb{N}$ be the sets of all integers and positive integers respectively. Let $d,k$ be fixed positive integers such that $\min\{d,k\}>1$ and $\gcd(d,k)=1$. The polynomial-exponential Diophantine equations of the form
\begin{equation}\label{eq.1.1}
x^2+d^y=k^z,\,\,x,y,z\in\mathbb{N}
\end{equation}
is usually called the generalized Ramanujan-Nagell equation. The solution of \eqref{eq.1.1} is an interesting problem with long history and rich contents (see \cite{LS}). In 2014, N. Terai \cite{T1} discussed the solution of \eqref{eq.1.1} in the case $d=2k-1$, and conjectured that, for any $k$ with $k>1$, the equation
\begin{equation}\label{eq.1.2}
x^2+(2k-1)^y=k^z,\,\,x,y,z\in\mathbb{N}
\end{equation}
has only one solution $(x,y,z)=(k-1,1,2)$.
The above conjecture has been verified in many special cases (see \cite[Theorem 3.1]{BB},\cite[Corollary 1.4]{DGX}, \cite[Corollary 1.1]{FL}, \cite[Theorem 1.2]{FT}, \cite{LS} and \cite[Proposition 3.3]{T1}). However, the case $4\mid k$ of this conjecture is a rather difficult problem. In this respect, N. Terai \cite{T1} used some classical number theory methods to discuss \eqref{eq.1.2} for $k\le 30$. However, his results are not available for $k\in\{12,24\}$. In 2017, M.A. Bennett and N. Billerey \cite{BB} used the modular approach to solve the case $k\in\{12,24\}$. Very recently, we follow similar method to solve \eqref{eq.1.2} for the case that $30<k<724$, $4\mid k$ and $2k-1$ is an odd prime power (see \cite[Theorem 1.1]{MLS}). In Section 3.2 of \cite{MLS}, using the theory of elliptic curves (e.g. the existence of $S$-integral points on a Weierstrass elliptic curve and descent methods on elliptic curves) and some elementary methods in number theory, we solve the Diophantine equations
\begin{equation}\label{eq.x.1}
	x^2+(2k-1)^3=k^z,\,\, z>3\,\,\text{odd},
\end{equation}
and
\begin{equation}\label{eq.x.2}
	x^2+(2k-1)^5=k^z,\,\, z>5\,\,\text{odd},	
\end{equation}
where $4\mid k$, $30<k <724$ and $2k-1$ is an odd prime power.

In this paper, by improving the results in Section 3.2 of \cite{MLS} and using various elementary methods, we prove the following results concerning the solutions $(x,y,z)$ of \eqref{eq.1.2} with $y\in\{3,5\}$.
\begin{theorem}\label{theo.1.1}
If $(2k-1)$ has a divisor $d$ with $d\equiv \pm 3 \pmod{8}$, then \eqref{eq.1.2} has no solutions $(x,y,z)$ with $y\in\{3,5\}$.
\end{theorem}
For any sufficiently large positive integer $N$, let $N_0$ denote the number of positive integers $k$ which satisfy the assumption of Theorem \ref*{theo.1.1} with $1\le k \le N$. Then we have 
\begin{equation*}
\lim_{N \to \infty} \dfrac{N_0}{N}\sim 1-\prod_{p}^{}(1-\dfrac{1}{p})\,\,\,(\text{p is prime with}\,\, p\equiv \pm 3\pmod{8}).
\end{equation*}
Thus for almost all positive integers $k$, the equation \eqref{eq.1.2} has no solutions $(x,y,z)$ with $y\in\{3,5\}$.
\begin{theorem}\label{theo.1.2}
If $k$ is a square, then \eqref{eq.1.2} has no solutions $(x,y,z)$ with $y\in\{3,5\}$.	
\end{theorem}
\begin{theorem}\label{theo.1.3}
If $k$ is not a square,	and $(x,y,z)$ is a solution of \eqref{eq.1.2} with $y\in\{3,5\}$, then $2\nmid z$ and
\begin{equation*}
y=Z_1t,\,\,t\in\mathbb{N},
\end{equation*}
\begin{equation*}
x+k^{(z-1)/2}\sqrt{k}=(X_1+\lambda Y_1\sqrt{k})^t(u+v\sqrt{k}),\,\,\lambda\in\{1,-1\},
\end{equation*}
where $X_1,Y_1,Z_1$ are positive integers such that
\begin{equation*}
X_1^2-kY_1^2=(-(2k-1))^{Z_1},\,\,\gcd(X_1,Y_1)=1,\,\,Z_1\mid h(4k)
\end{equation*}
and
\begin{equation*}
1<\left\lvert \dfrac{X_1+Y_1\sqrt{k}}{X_1-Y_1\sqrt{k}}\right\rvert<u_1+v_1\sqrt{k},
\end{equation*}
where $h(4k)$ is the class number of binary quadratic primitive forms with discriminant $4k$, $(u,v)$ is a solution of Pell's equation
\begin{equation}\label{eq.1.3}
u^2-kv^2=1,\,\,u,v\in\mathbb{Z},
\end{equation}
and $(u_1,v_1)$ is the least solution of \eqref{eq.1.3}.
\end{theorem}
Obviously, from Theorem \ref{theo.1.3}, we can derive several criteria for determining whether \eqref{eq.1.2} has solutions $(x,y,z)$ with $y\in\{3,5\}$ for specific values of $k$.

\section{Preliminaries}
\begin{lemma}\label{lem.2.1}
Let $F(t)=t+a/t$ be a function of the real variable $t$, where $a$ is a constant with $a>1$. Then  $F(t)$ is a strictly decreasing function for $1\le t <\sqrt{a}$.
\end{lemma}
\begin{proof}
Since $F'(t)=1-a/t^2<0$ for $1\le t<\sqrt{a}$, where $F'(t)$ is the derivative of $F(t)$, we obtain the lemma immediately.
\end{proof}
We use the notation $\mathbb{Z}[t]$ for the set of all the polynomials of indeterminate $t$ with integer coefficients. It is a well known fact that if $F(t)\in\mathbb{Z}[t]$ which leading coefficient is positive, then there exist positive integers $m$ which can make
\begin{equation}\label{eq.2.1}
F(t)\in\mathbb{N},\,\,t\in\mathbb{N},\,\,t\ge m.
\end{equation}
Therefore, we may use the notation $m(F(t))$ to represent the least value of positive integers $m$ with \eqref{eq.2.1}.
\begin{lemma}\label{lem.2.2}
Let $F(t)=t^{2n}-a_{2n-1}t^{2n-1}-\cdots -a_0\in\mathbb{Z}[t]$, where $n$ is a positive integer. If there exist $G(t),R(t)\in\mathbb{Z}[t]$ such that
\begin{equation}\label{eq.2.2}
F(t)=(G(t))^2+R(t),
\end{equation}
where
\begin{equation}\label{eq.2.3}
\begin{aligned}
&G(t)=t^n-b_{n-1}t^{n-1}-\cdots -b_0,\,\,R(t)=r_{\ell}t^{\ell}-r_{\ell-1}t^{\ell-1}\\
&-\cdots -r_0,\,\,
r_{\ell}\neq 0,\,\,\ell<n,
\end{aligned}
\end{equation}
then the equation
\begin{equation}\label{eq.2.4}
X^2=F(Y),\,\,X,Y\in\mathbb{N}
\end{equation}
has no solutions $(X,Y)$ with $Y\ge Y_0$, where
\begin{equation}\label{eq.2.5}
Y_0=
	\begin{cases}
		\max\{m(G(t)),\,\,m(R(t)),\,\,m(2G(t)-R(t))\}, \qquad \quad \, \textrm{if $r_{\ell}>0$},\\
		\max\{m(G(t)),\,\,m(-R(t)),\,\,m(2G(t)+R(t)-1)\}, \quad  \textrm{if $r_{\ell}<0$}.
	\end{cases}
\end{equation} 
\end{lemma}
\begin{proof}
We now assume that $(X,Y)$ is a solution of \eqref{eq.2.4} with $Y\ge Y_0$. By \eqref{eq.2.2} and \eqref{eq.2.4}, we have
\begin{equation}\label{eq.2.6}
X^2=(G(Y))^2+R(Y).
\end{equation}

When $r_{\ell}>0$, by \eqref{eq.2.5}, we have $Y\ge \max\{m(G(t)),m(R(t))\}$. It implies that $G(Y)$ and $R(Y)$ are positive integers. Hence, by \eqref{eq.2.6}, we have
\begin{equation}\label{eq.2.7}
X+G(Y)=A,\,\,X-G(Y)=B,
\end{equation}
where
\begin{equation}\label{eq.2.8}
R(Y)=AB,\,\,A,B\in\mathbb{N},\,\,A>B.
\end{equation}
Further, by \eqref{eq.2.7} and \eqref{eq.2.8}, we get
\begin{equation*}
R(Y)=AB\ge A=X+G(Y)=2G(Y)+B>2G(Y),
\end{equation*}
whence we obtain
\begin{equation}\label{eq.2.9}
2G(Y)-R(Y)<0.
\end{equation}
Recall that $\ell<n$. We see from \eqref{eq.2.3} that $2G(t)-R(t)\in\mathbb{Z}[t]$ which has positive leading coefficient. In addition, by \eqref{eq.2.5}, we have $Y\ge m(2G(t)-R(t))$ and $2G(Y)-R(Y)>0$, which contradicts \eqref{eq.2.9}.

When $r_{\ell}<0$, by \eqref{eq.2.3}, we have $-R(t)\in\mathbb{Z}[t]$ which leading coefficient is positive. Hence, by \eqref{eq.2.4} and \eqref{eq.2.5}, $X$, $G(Y)$ and $-R(Y)$ are positive integers. By \eqref{eq.2.6}, we have
\begin{equation*} 
(G(Y))^2-X^2=-R(Y)		
\end{equation*}
and
\begin{equation}\label{eq.2.10}
G(Y)+X=A,\,\,G(Y)-X=B,
\end{equation}
where
\begin{equation}\label{eq.2.11}
-R(Y)=AB,\,\,A,B\in\mathbb{N},\,\,A>B.
\end{equation}
Eliminating $X$ from \eqref{eq.2.10}, by \eqref{eq.2.11}, we get
\begin{equation}\label{eq.2.12}
2G(Y)=A+B=\dfrac{-R(Y)}{B}+B.
\end{equation}
Take $a=-R(Y)$ and $t=B$. By Lemma \ref{lem.2.1}, we have
\begin{equation}\label{eq.2.13}
\dfrac{-R(Y)}{B}+B\le -R(Y)+1.
\end{equation}
By \eqref{eq.2.12} and \eqref{eq.2.13}, we get
\begin{equation}\label{eq.2.14}
2G(Y)+R(Y)-1\le 0.
\end{equation}
However, since $Y\ge m(2G(t)+R(t)-1)$ by \eqref{eq.2.5}, we have $2G(Y)+R(Y)-1>0$, which contradicts \eqref{eq.2.14}. To sum up, \eqref{eq.2.4} has no solutions $(X,Y)$ with $Y\ge Y_0$. The lemma is proved. 
\end{proof}
\begin{lemma}\label{lem.2.3}
Each of the following equations has no solutions $(X,Y)$.
\begin{equation}\label{eq.2.15}
	X^2=Y^4-8Y^3+12Y^2-6Y+1, \ \ X,Y \in \mathbb{N}.
\end{equation}
\begin{equation}\label{eq.2.16}
	X^2=Y^6-8Y^3+12Y^2-6Y+1, \ \ X,Y \in \mathbb{N}.
\end{equation}
\begin{equation}\label{eq.2.17}
	X^2=Y^6-32Y^5+80Y^4-80Y^3+40Y^2-10Y+1, \ \ X,Y \in \mathbb{N}.
\end{equation}
\begin{equation}\label{eq.2.18}
	X^2=Y^8-32Y^5+80Y^4-80Y^3+40Y^2-10Y+1, \ \ X,Y \in \mathbb{N}.
\end{equation}
\begin{equation}\label{eq.2.19}
	X^2=Y^{10}-8Y^6+12Y^4-6Y^2+1, \ \ X,Y \in \mathbb{N}.
\end{equation}
\begin{equation}\label{eq.2.20}
	X^2=Y^{10}-32Y^5+80Y^4-80Y^3+40Y^2-10Y+1, \ \ X,Y \in \mathbb{N}.
\end{equation}
\begin{equation}\label{eq.2.21}
	X^2=Y^{14}-32Y^{10}+80Y^8-80Y^6+40Y^4-10Y^2+1, \ \ X,Y \in \mathbb{N}.
\end{equation}
\begin{equation}\label{eq.2.22}
	X^2=Y^{18}-32Y^{10}+80Y^8-80Y^6+40Y^4-10Y^2+1, \ \ X,Y \in \mathbb{N}.
\end{equation}
\end{lemma}
\begin{proof}
First, we consider the solution of \eqref{eq.2.15}. We check by direct verification that \eqref{eq.2.15} has no solutions $(X,Y)$ with $Y<16$. On the other hand, let $F(t)=t^4-8t^3+12t^2-6t+1, G(t)=t^2-4t-2$ and $R(t)=-22t-3$. Then $F(t),G(t)$ and $R(t)$ satisfy \eqref{eq.2.2}. Since $m(G(t))=5, m(-R(t))=1$ and $m(2G(t)+R(t)-1)=16$, by Lemma \ref{lem.2.2}, \eqref{eq.2.15} has no solutions $(X,Y)$ with $Y\geq16$. Therefore, the lemma is true for \eqref{eq.2.15}.

Using the same method as in the above, we can solve other seven equations. The following is the information needed to solve the seven equations.
\begin{itemize}
	\item[(i)] $F(t)=t^6-8t^3+12t^2-6t+1,\ \ G(t)=t^3-4, \ \ R(t)=12t^2-6t-15,\ \ m(G(t))=2,\ \ m(R(t))=2,\ \ m(2G(t)-R(t))=1$.
	\\
	\item[(ii)] $F(t)=t^6-32t^5+80t^4-80t^3+40t^2-10t+1,\ \ G(t)=t^3-16t^2-88t-1448,\ \ R(t)=-54040t^2-254858t-2096703,\ \ m(G(t))=23,\ \ m(-R(t))=1,\ \ m(2G(t)+R(t)-1)=27041$.
	\\
	\item[(iii)] $F(t)=t^8-32t^5+80t^4-80t^3+40t^2-10t+1,\ \ G(t)=t^4-16t+40,\ \ R(t)=-80t^3-216t^2+1270t-1599,\ \ m(G(t))=1,\ \ m(-R(t))=1,\ \ m(2G(t)+R(t)-1)=43$.
	\\
	\item[(iv)] $F(t)=t^{10}-8t^6+12t^4-6t^2+1,\ \ G(t)=t^5-4t,\ \ R(t)=12t^4-22t^2+1,\ \ m(G(t))=2,\ \ m(R(t))=2,\ \ m(2G(t)-R(t))=1$.
	\\
	\item[(v)] $F(t)=t^{10}-32t^5+80t^4-80t^3+40t^2-10t+1,\ \ G(t)=t^5-16,\ \ R(t)=80t^4-80t^3+40t^2-10t-255,\ \ m(G(t))=2,\ \ m(R(t))=2,\ \ m(2G(t)-R(t))=1$.
	\\
	\item[(vi)] $F(t)=t^{14}-32t^{10}+80t^8-80t^6+40t^4-10t^2+1,\ \ G(t)=t^7-16t^3+40t,\ \ R(t)=-336t^6+1320t^4-1610t^2+1,\ \ m(G(t))=1,\ \ m(-R(t))=1,\ \ m(2G(t)+R(t)-1)=168$.
	\\
	\item[(vii)] $F(t)=t^{18}-32t^{10}+80t^8-80t^6+40t^4-10t^2+1,\ \ G(t)=t^9-16t,\ \ R(t)=80t^8-80t^6+40t^4-266t^2+1,\ \ m(G(t))=2,\ \ m(R(t))=2,\ \ m(2G(t)-R(t))=1$.
\end{itemize}
Thus, the lemma is proved.
\end{proof}

\begin{lemma}\label{lem.2.4}
If $(x,y,z)$ is a solution of \eqref{eq.1.2} with $y\in\{3,5\}$, then $2\nmid z$. 
\end{lemma}
\begin{proof}
We now assume that $(x,z)$ is a solution of the equation
\begin{equation}\label{eq.2.23}
x^2+(2k-1)^3=k^z,\,\,x,z\in\mathbb{N}
\end{equation}
with $2\mid z$. Since $k>1$, we get $(2k-1)^3>k^2$. Then we have $z>3$. Since $2\nmid 2k-1$ and $\gcd(k,2k-1)=1$, by \eqref{eq.2.23}, we get
\begin{equation}\label{eq.2.24}
k^{z/2}+x=f^3,\,\,k^{z/2}-x=g^3,
\end{equation}
where 
\begin{equation}\label{eq.2.25}
2k-1=fg,\,\,f,g\in\mathbb{N},\,\,f>g,\,\,\gcd(f,g)=1,\,\,2\nmid fg.	
\end{equation}
Eliminating $x$ from \eqref{eq.2.24}, by \eqref{eq.2.25}, we have
\begin{equation}\label{eq.2.26}
2k^{z/2}=f^3+g^3=\big(\dfrac{2k-1}{g}\big)^3+g^3.
\end{equation}
take $a=(2k-1)^3$ and $t=g^3$. By Lemma \ref*{lem.2.1}, we have
\begin{equation}\label{eq.2.27}
\big(\dfrac{2k-1}{g}\big)^3+g^3\le (2k-1)^3+1.
\end{equation}
Hence, by \eqref{eq.2.26} and \eqref{eq.2.27}, we get
\begin{equation}\label{eq.2.28}
2k^{z/2}\le (2k-1)^3+1<8k^3.
\end{equation}
Further, since $k>1$, $z>3$ and $2\mid z$, by \eqref{eq.2.28}, we obtain $2\le z/2\le 3$ and $z\in\{4,6\}$.

When $z=4$, by \eqref{eq.2.23}, we have
\begin{equation}\label{eq.2.29}
x^2=k^4-8k^3+12k^2-6k+1.
\end{equation}
We see from \eqref{eq.2.29} that \eqref{eq.2.15} has a solution $(X,Y)=(x,k)$. But, by Lemma \ref*{lem.2.3}, it is impossible.

Similarly, when $z=6$, we find from \eqref{eq.2.23} that \eqref{eq.2.16} has a solution $(X,Y)=(x,k)$. However, by Lemma \ref*{lem.2.3}, it is also impossible. Therefore, \eqref{eq.1.2} has no solutions $(x,y,z)$ with $y=3$ and $2\mid z$.

If $(x,y,z)$ is a solution of \eqref{eq.1.2} with $y=5$ and $2\mid z$, then we have
\begin{equation}\label{eq.2.30}
x^2+(2k-1)^5=k^z,\,\,x,z\in\mathbb{N}
\end{equation}
and
\begin{equation}\label{eq.2.31}
k^{z/2}+x=f^5,\,\,k^{z/2}-x=g^5,
\end{equation} 
where $f$ and $g$ satisfy \eqref{eq.2.25}. Further, by \eqref{eq.2.25} and \eqref{eq.2.31}, we have
\begin{equation*}
2k^{z/2}=f^5+g^5=\big(\dfrac{2k-1}{g}\big)^5+g^5\le (2k-1)^5+1<32k^5,
\end{equation*}
whence we get $3\le z/2\le 5$ and $z\in\{6,8,10\}$. Hence, by \eqref{eq.2.30}, the equations \eqref{eq.2.17}, \eqref{eq.2.18} or \eqref{eq.2.20} has a solution $(X,Y)=(x,k)$ according as $z=6,8$ or 10. But by Lemma \ref*{lem.2.3}, it is impossible. Thus, the lemma is proved.
\end{proof}

Let $D$ be a fixed nonsquare positive integer, and let $h(4D)$ denote the class number of binary quadratic primitive forms with discriminant $4D$. Further let $K$ be a fixed odd integer with $|K|>1$ and $\gcd(D,K)=1$. It is well known that Pell's equation
\begin{equation}\label{eq.2.32}
	U^2-DV^2=1,\,\,U,V\in\mathbb{Z}
\end{equation}  
has positive integer solutions $(U,V)$, and it has a unique positive integer solution $(U_1,V_1)$ such that $U_1+V_1\sqrt{D}\le U+V\sqrt{D}$, where $(U,V)$ through all positive integer solutions of \eqref{eq.2.32}. The solution $(U_1,V_1)$ is called the least solution of \eqref{eq.2.32}. For any positive integer $n$, let
\begin{equation*}
	U_n+V_n\sqrt{D}=(U_1+V_1\sqrt{D})^n.
\end{equation*}
Then $(U,V)=(U_n,V_n)$ $(n=1,2,\cdots)$ are all positive integer solutions of \eqref{eq.2.32}. It follows that every solution $(U,V)$ of \eqref{eq.2.32} can be expressed as
\begin{equation}\label{eq.2.33}
	U+V\sqrt{D}=\lambda_1(U_1+\lambda_2V_1\sqrt{D})^m,\,\,\lambda_1,\lambda_2\in\{1,-1\},\,\,m\in\mathbb{Z},\,\,m\ge 0.
\end{equation}
Hence, by \eqref{eq.2.33}, every solution $(U,V)$ of \eqref{eq.2.32} satisfies
\begin{equation}\label{eq.2.34}
V \equiv 0 \pmod{V_1}.
\end{equation}
\begin{lemma}(\cite{Le},\cite{YanFu})\label{lem.2.5}
If the equation
\begin{equation}\label{eq.2.35}
X^2-DY^2=K^Z,\,\,X,Y,Z\in\mathbb{Z},\,\,\gcd(X,Y)=1,\,\,Z>0
\end{equation}
has solutions $(X,Y,Z)$, then every solution $(X,Y,Z)$ of \eqref{eq.2.35} can be expressed as
\begin{equation*}
	Z=Z_1t,\,\,t\in\mathbb{N},
\end{equation*}
\begin{equation*}
	X+Y\sqrt{D}=(X_1+\lambda Y_1\sqrt{D})^t(U+V\sqrt{D}),\,\,\lambda\in\{1,-1\},
\end{equation*}
where $(U,V)$ is a solution of \eqref{eq.2.32} and $X_1,Y_1,Z_1$ are positive integers satisfy
\begin{equation}\label{eq.2.36}
X_1^2-DY_1^2=K^{Z_1},\,\,\gcd(X_1,Y_1)=1,\,\,Z_1\mid h(4D)
\end{equation}
and
\begin{equation}\label{eq.2.37}
1<\left\lvert \dfrac{X_1+Y_1\sqrt{D}}{X_1-Y_1\sqrt{D}}\right\rvert<U_1+V_1\sqrt{D},
\end{equation}
and $(U_1,V_1)$ is the least solution of \eqref{eq.2.32}.
\end{lemma}

\section{Proofs of Theorems}
\begin{proof}[Proof of Theorem \ref{theo.1.1}]
We now assume that $(x,y,z)$ is a solution of \eqref{eq.1.2} with $y\in\{3,5\}$. By Lemma \eqref{lem.2.4}, we have $2\nmid z$. Hence, for any divisor $d$ of $2k-1$, we get from \eqref{eq.1.2} that
\begin{equation}\label{eq.3.1}
1=\bigg(\dfrac{k^z}{d}\bigg)=\bigg(\dfrac{k}{d}\bigg)=\bigg(\dfrac{4k}{d}\bigg)=\bigg(\dfrac{2}{d}\bigg),
\end{equation}
where $(*/*)$ is the Jacobi symbol. However, if $d\equiv \pm 3 \pmod{8}$, then we have $(2/d)=-1$, which contradicts \eqref{eq.3.1}. Thus, the theorem is proved.
\end{proof}	

\bigskip

\begin{proof}[Proof of Theorem \ref{theo.1.2}]
Since $k$ is square, we have
\begin{equation}\label{eq.3.2}
k=\ell^2,\,\,\ell\in\mathbb{N}.
\end{equation}
Substituting \eqref{eq.3.2} into \eqref{eq.2.23} and \eqref{eq.2.30}, we get
\begin{equation}\label{eq.3.3}
x^2+(2\ell^2-1)^3=\ell^{2z},\,\,x,z\in\mathbb{N},\,\,z>3
\end{equation}
and
\begin{equation}\label{eq.3.4}
x^2+(2\ell^2-1)^5=\ell^{2z},\,\,x,z\in\mathbb{N},\,\,z>5,
\end{equation}
respectively. In addition, by Lemma \ref*{lem.2.4}, we have $2\nmid z$.

If $(x,z)$ is a solution of \eqref{eq.3.3}, then we have
\begin{equation}\label{eq.3.5}
\ell^z+x=f^3,\,\,\ell^z-x=g^3,	
\end{equation}
where
\begin{equation}\label{eq.3.6}
2\ell^2-1=fg,\,\,f,g\in\mathbb{N},\,\,f>g,\,\,\gcd(f,g)=1,\,\,2\nmid fg.		
\end{equation}
Eliminating $x$ from \eqref{eq.3.5}, by \eqref{eq.3.6} and Lemma \ref{lem.2.1}, we get
\begin{equation}\label{eq.3.7}
2\ell^z=f^3+g^3=\big(\dfrac{2\ell^2-1}{g}\big)^3+g^3\le (2\ell^2-1)^3+1<8\ell^6.
\end{equation}
Recall that $2\nmid z$. By \eqref{eq.3.7}, we have $3<z\le 6$ and so $z=5$. Hence, we see from \eqref{eq.3.3} that \eqref{eq.2.19} has a solution $(X,Y)=(x,\ell)$. But, by Lemma \ref*{lem.2.3}, it is impossible. Therefore, if $k$ is a square, then \eqref{eq.1.2} has no solutions $(x,y,z)$ with $y=3$.

Similarly, if $(x,z)$ is a solution of \eqref{eq.3.4}, then
\begin{equation}\label{eq.3.8}
\ell^z+x=f^5,\,\,\ell^z-x=g^5,
\end{equation} 
where $f$ and $g$ satisfy \eqref{eq.3.6}. Further, by \eqref{eq.3.6}, \eqref{eq.3.8} and Lemma \ref*{lem.2.1}, we have
\begin{equation*}
2\ell^z=f^5+g^5=\big(\dfrac{2\ell^2-1}{g}\big)^5+g^5\le (2\ell^2-1)^5+1<32\ell^{10},
\end{equation*}
whence we get $5<z\le 10$ and so $z\in\{7,9\}$. Hence, by \eqref{eq.3.4}, the equation \eqref{eq.2.21} or \eqref{eq.2.22} has a solution $(X,Y)=(x,\ell)
$ according as $z=7$ or 9. But, by Lemma \ref*{lem.2.3}, it is also impossible. Thus, the theorem is proved.
\end{proof}

\begin{proof}[Proof of Theorem \ref{theo.1.3} ]
Notice that $k$ is not a square, $2\nmid 2k-1$, $\gcd(k,2k-1)=1$ and $2\nmid z$ by Lemma \ref{lem.2.4}, then the equation
\begin{equation}\label{eq.3.9}
X^2-kY^2=(-(2k-1))^Z,\,\,X,Y,Z\in \mathbb{Z},\,\,\gcd(X,Y)=1,\,\,Z>0
\end{equation}
has a solution
\begin{equation}\label{eq.3.10}
(X,Y,Z)=(x,k^{(z-1)/2},y).
\end{equation}
Therefore, applying Lemma \ref{lem.2.5} to \eqref{eq.3.9} and \eqref{eq.3.10}, we can obtain the theorem immediately.
\end{proof}

\section{An application of Theorem \ref{theo.1.3}}
Now we illustrate Theorem \ref{theo.1.3} for determining whether \eqref{eq.1.2} has solutions with $y\in\{3,5\}$ for $k=736$. These two cases correspond to the Diophantine equations
\begin{equation}\label{eq.4.1}
	x^2+1471^3=736^{z},\,\,x,z\in\mathbb{N},	
\end{equation}
\begin{equation}\label{eq.4.2}
	x^2+1471^5=736^{z},\,\,x,z\in\mathbb{N},	
\end{equation}
respectively. Here we will only solve the equation \eqref{eq.4.1}. Equation \eqref{eq.4.2} can be treated similarly.

We now to prove that \eqref{eq.4.1} has no solutions $(x,z)$. To do this, we use Lemma \ref{lem.2.5}.

If $(x,z)$ is a solution of \eqref{eq.4.1}, then the equation
\begin{equation}\label{eq.4.3}
	X^2-736Y^2=(-1471)^Z,\,\,X,Y,Z\in\mathbb{Z},\,\,\gcd(X,Y)=1,\,\,Z>0
\end{equation} 
has a solution
\begin{equation}\label{eq.4.4}
	(X,Y,Z)=(x,736^{(z-1)/2},3).
\end{equation}
Let $(X_1,Y_1,Z_1)$ be a solution of \eqref{eq.4.3}. Applying Lemma \ref{lem.2.5} to \eqref{eq.4.4}, we have
\begin{equation}\label{eq.4.5}
	3=Z_1t,\,\,t\in\mathbb{N}.
\end{equation}
On the other hand, since $h(4\times 736)=4$, by \eqref{eq.2.36} we have $4\equiv 0\pmod{Z_1}$. Hence, we see from \eqref{eq.4.5} that $Z_1=1$. So we have
\begin{equation}\label{eq.4.6}
	X_1^2-736Y_1^2=-1471,\,\,X_1,Y_1\in\mathbb{N},\,\,\gcd(X_1,Y_1)=1.
\end{equation}
Further, since the least solution of Pell's equation
\begin{equation*}
	U^2-736V^2=1,\,\,U,V\in\mathbb{Z}
\end{equation*}
is $(U_1,V_1)=(24335,897)$, by \eqref{eq.2.37}, we have
\begin{equation}\label{eq.4.7}
	1<\left\lvert \dfrac{X_1+Y_1\sqrt{736}}{X_1-Y_1\sqrt{736}}\right\rvert<24335+897\sqrt{736}.
\end{equation}
Hence, by \eqref{eq.4.6} and \eqref{eq.4.7}, we get
\begin{equation}\label{eq.4.8}
	X_1+Y_1\sqrt{736}<\sqrt{1471(24335+897\sqrt{736})}<8462.
\end{equation} 
Using MAPLE \cite{Celik}, by Lemma \ref{lem.2.5}, we see that the only solution $(X_1,Y_1,Z_1)$ of \eqref{eq.4.6} is $(2577,95,1)$. Then we have 
\begin{equation}\label{eq.4.9}
	x+736^{(z-1)/2}\sqrt{736}=(2577+95\lambda\sqrt{736})^3(U+V\sqrt{736}),\,\,\lambda\in\{1,-1\},
\end{equation}
where $(U,V)$ is a solution of Pell's equation
\begin{equation}\label{eq.4.10}
	U^2-736V^2=1,\,\,U,V\in\mathbb{Z}.
\end{equation}
Let
\begin{equation}\label{eq.4.11}
	f+g\lambda\sqrt{736}=(2577+95\lambda\sqrt{736})^3.
\end{equation}
Obviously, $f$ and $g$ are positive integers. Substitute \eqref{eq.4.11} into \eqref{eq.4.9}, we have
\begin{equation*}
	x+736^{(z-1)/2}\sqrt{736}=(f+g\lambda\sqrt{736})(U+V\sqrt{736}),
\end{equation*}
whence we get
\begin{equation}\label{eq.4.12}
	736^{(z-1)/2}=fV+\lambda gU.
\end{equation}
Since the least solution of \eqref{eq.4.10} is $(U_1,V_1)=(24335,897)$, by \eqref{eq.2.34}, we have $V\equiv 0 \pmod{897}$. So we obtain $23\mid V$. Hence, by \eqref{eq.4.12}, we get
\begin{equation}\label{eq.4.13}
	0\equiv \lambda gU \pmod{23}.
\end{equation}
Further, since $\lambda\in\{1,-1\}$ and $\gcd(U,23)=1$ by \eqref{eq.4.10}, we see from \eqref{eq.4.13} that
\begin{equation}\label{eq.4.14}
	g \equiv 0 \pmod{23}.
\end{equation}
However, by \eqref{eq.4.11}, we have $g=2523692765\equiv 9 \pmod{23}$. It implies that \eqref{eq.4.14} is false. Therefore, \eqref{eq.4.1} has no solutions $(x,z)$.

\section*{Acknowledgments} 
We would like to thank Professor Nikos Tzanakis for useful discussions and anonymous referee for carefully reading our paper and for his/her corrections. The first author is supported by the Scientific and Technological Research Council of Turkey (T\"UB\.ITAK) 2211/A National PhD scholarship program.

\end{document}